\documentclass[twoside,a4paper,12pt]{article}
\usepackage{amssymb, amsmath, amsthm}
\usepackage[utf8]{inputenc}
\usepackage{color}
\usepackage{url}

\newdimen\myMargin
\textwidth \paperwidth
\advance\textwidth -2\myMargin
\textheight \paperheight
\advance\textheight -2\myMargin
\advance\oddsidemargin \myMargin
\advance\evensidemargin \myMargin
\advance\topmargin \myMargin
\advance\textheight -17 true mm
\advance\topmargin 12.5 true mm

\newcommand{\A}{{\cal A}}
\def\B{{\cal B}}

\newcommand{\Fq}{\mathbb F_q}
\newcommand{\F}{\mathbb F}

\newtheorem{theorem}{Theorem}
\newtheorem{lemma}[theorem]{Lemma}
\newtheorem{corollary}[theorem]{Corollary}
\newtheorem{proposition}[theorem]{Proposition}
\newtheorem{definition}[theorem]{Definition}

\newtheorem{problem}[theorem]{Problem}

\theoremstyle{remark}
\theoremstyle{remark}\newtheorem{remark}[theorem]{Remark}

\title{Primitive idempotents in central simple algebras over $\mathbb{F}_q(t)$ with an application to coding theory\footnote{Research partially supported by grant {MTM2016-78364-P} from {Agencia Estatal de Investigaci\'{o}n (AEI) and from Fondo Europeo de Desarrollo Regional (FEDER)}}}
\author{J. G\'omez-Torrecillas\footnote{CITIC and Department of Algebra, 
University of Granada, E18071, Granada, Spain}  \and 
P. Kutas\footnote{School of Computer Science, University of Birmingham, B15 2TT, Birmingham, United Kingdom} \and 
F. J. Lobillo\footnote{CITIC and Department of Algebra, 
University of Granada, E18071, Granada, Spain} \and 
G. Navarro\footnote{CITIC and Department of Computer Science and A. I., 
University of Granada, E18071, Granada, Spain} }

\begin{document}

\maketitle

\begin{abstract}
We consider the algorithmic problem of computing a primitive idempotent of a central simple algebra over the field of rational functions over a finite field. The algebra is given by a set of structure constants. The problem is reduced to the computation of a division algebra Brauer equivalent to the central simple algebra. This division algebra is constructed as a cyclic algebra, once the Hasse invariants have been computed. We give an application to skew constacyclic convolutional codes. 

\end{abstract}

\section{Introduction}

We consider the following algorithmic problem. Let $\F_q$ be the finite field with $q$ elements and let $\A$ be a central simple algebra over $\F_q(t)$ (the field of rational functions in the variable $t$) with $\F_q(t)$-basis $b_1,\dots,b_{n^2}$. Then one has that 
$$b_ib_j=\sum_{k=1}^{n^2}\gamma_{ijk}b_k$$
where $\gamma_{ijk}\in\F_q(t)$. The $\gamma_{ijk}$ are called structure constants. We consider $\A$ to be given as a collection of structure constants. The task is to find a primitive idempotent in $\A$. This problem is closely related to the factorization problem of Ore polynomials with  coefficients in $\F_q(t)$ \cite{GLN3}. In \cite{IKR} the split case, namely where $\A\cong M_n(\F_q(t))$,  is studied. Here we investigate the problem when $\A\cong M_k(D)$ where $D$ is a division algebra over $\F_q(t)$. 

In Section \ref{sec:isomorphism} we reduce the problem of computing a primitive idempotent in $\A$ to computing $D$, the division algebra Brauer equivalent to $\A$ (building on the algorithm from \cite{IKR}). A central simple algebra over $\F_q(t)$ is determined (up to Brauer equivalence) by its Hasse invariants (see, e.g., \cite[Corollary 6.5.4]{Gille/Szamuely:2006}). This means that by computing the Hasse invariants of $\A$ and constructing a division algebra with those invariants provides a method for calculating the underlying division algebra of $\A$. 

In \cite{BG} the authors propose a randomized polynomial time algorithm for constructing an $\F_q(t)$-division algebra provided that the invariant at infinity is zero and the degree of the algebra is coprime to $q$. We propose an algorithm where the invariant at infinity is not necessarily zero when $\F_q$ contains the $n$th roots of unity. Here $n$ is the degree of the division algebra. When the degree of $\A$ is coprime to $q$, these algorithms can also be used to compute the Hasse invariants of $\A$. 

We also give an application of our results. Linear convolutional codes of length $n$ can be modeled (see \cite{Forney:1970, Johannesson/Zingangirov:1999, Rosenthal/Smarandache:1999}) as vector subspaces of $\F_q(t)^n$, where the variable $t$ represents the delay operator.  Based on this model, an approach to cyclic convolutional codes was proposed in \cite{Gomez/alt:2016}. So, given an automorphism $\sigma$ of $\F_q(t)$,  a skew cyclic convolutional code is a left ideal of a cyclic algebra $(\F_q(t),\sigma,1)$, endowed with the Hamming metric induced by the natural basis of $(\F_q(t),\sigma,1)$. These codes became MDS for the Hamming distance, and efficient algebraic decoding algorithms were designed for them \cite{GLN2,Gomez/alt:2018a}.  A natural question is what can be done when the skew cyclic structure is given by a cyclic algebra of the form $(\F_q(t),\sigma,\lambda)$, for a general $\lambda \in \Fq(t)^\sigma$.  This lead to the notion of skew constacyclic code. We will show that, if we know an explicit algebra isomorphism  $(\F_q(t),\sigma,\lambda) \cong M_k(D)$, where $D$ is a division algebra over $\F_q(t)^\sigma$, then the construction of skew convolutional codes and the decoding algorithms from \cite{Gomez/alt:2016, GLN2} can be extended to skew constacyclic convolutional codes.

The structure of the paper is as follows. In Section 2 we recall some basic facts about quaternion and symbol algebras. In Section 3 we provide randomized polynomial time algorithms for computing quaternion and symbol algebras with given invariants. In Section 4 we show how one can compute a division algebra $D$ Brauer equivalent to a given central simple $\F_q(t)$-algebras, and an explicit isomorphism with the corresponding matrix ring over $D$, using either the algorithms from Section 3 or the algorithm from \cite{BG}. In Section 5 we construct constacyclic convolutional codes of designed Hamming distance and propose a polynomial time decoding algorithm.

\section{Quaternion and symbol algebras with prescribed invariants}

 Let $K$ be a field such that the multiplicative group $K^*$ contains a cyclic group of order $n$, and let $\epsilon \in K$ be a primitive $n$--th root of unity $\epsilon$. Choose $a, b \in K^*$. The \emph{symbol algebra} (or power residue algebra) $(a,b;K,\epsilon)$ is the $K$--algebra with generators $u, v$ subject to the relations 

$$u^n=a,~ v^n=b,~ uv=\epsilon vu.$$ 
When $n = 2$ (and, hence, $K$ must be of characteristic different from $2$), symbol algebras are called \emph{quaternion algebras}.

Symbol algebras are central simple $K$--algebras \cite[Chapter 11, Theorem 1]{Draxl}.

\subsection{Quaternion algebras}
In this subsection we propose a randomized polynomial-time algorithm which constructs a quaternion algebra over $\F_q(t)$ with $q$ an odd prime power which ramifies at prescribed places. 

First we cite an estimate on the number of irreducible polynomials in a given residue class. This is an analogue of Dirichlet's theorem on primes in arithmetic progressions. However, in the function field case a much stronger result is true:

\begin{proposition}\cite[Theorem 5.1.]{Wan}\label{number}
Let $a,m\in\mathbb{F}_q[t]$ be such that $deg(m)>0$ and the $gcd(a,m)=1$. Let $N$ be a positive integer and let  
$$S_N(a,m)=\#\{f\in\mathbb{F}_q[t]~\text{monic irreducible}~|~f\equiv a ~(mod ~m),~ deg(f)=N\}.$$
Let $M=deg(m)$ and let $\Phi(m)$ denote the number of polynomials in $\mathbb{F}_q[t]$ relative prime to $m$ whose degree is smaller than M. Then we have the following inequality:
\begin{equation*}
|S_N(a,m)-\frac{q^N}{\Phi(m)N}|\leq \frac{1}{N}(M+1)q^{\frac{N}{2}}.
\end{equation*}
\end{proposition}

We also state two lemmas from \cite{IKR2}.
\begin{lemma}\label{Lemma1}\cite[Lemma 6]{IKR2}
 Let $a_1,a_2,a_3\in\mathbb{F}_q[t]$ be nonzero polynomials. Let $f$ be a monic irreducible polynomial. Let $\mathbb{F}_q(t)_{(f)}$ denote the $f$-adic completion of $\mathbb{F}_q(t)$. Let $v_f(a_i)$ denote the multiplicity of $f$ in the prime decomposition of $a_i$. Then the following hold:
\begin{enumerate}
\item If $v_f(a_1)\equiv v_f(a_2)\equiv v_f(a_3) ~(mod~ 2)$ then the equation $a_1x_1^2+a_2x_2^2+a_3x_3^2=0$ is solvable in $\mathbb{F}_q(t)_{(f)}$.
\item Assume that not all the $v_f(a_i)$ have the same parity. Also suppose that $v_f(a_i)\equiv v_f(a_j) ~(mod~ 2)$. Then the equation $a_1x_1^2+a_2x_2^2+a_3x_3^2=0$ is solvable in $\mathbb{F}_q(t)_{(f)}$ if and only if $-f^{-v_f(a_ia_j)}a_ia_j$ is a square modulo $f$.
\end{enumerate}
\end{lemma}

\begin{lemma}\label{Lemma1'}\cite[Lemma 10]{IKR2}
Let $a_1,a_2,a_3\in\mathbb{F}_q[t]$ be nonzero polynomials. Then the following hold:
\begin{enumerate}
\item If the degrees of the $a_i$ all have the same parity then the equation $a_1x_1^2+a_2x_2^2+a_3x_3^2=0$ admits a nontrivial solution in $\mathbb{F}_q((\frac{1}{t}))$.
\item Assume that not all of the degrees of the $a_i$ have the same parity. Also assume that $deg(a_i)\equiv deg(a_j)~(mod~ 2)$. Let $c_i$ and $c_j$ be the leading coefficients of $a_i$ and $a_j$ respectively. Then the equation $a_1x_1^2+a_2x_2^2+a_3x_3^2=0$ has a nontrivial solution in $\mathbb{F}_q((\frac{1}{t}))$ if and only if $-c_ic_j$ is a square in $\mathbb{F}_q$.
\end{enumerate}
\end{lemma}

Assume $S$ is a set of places of $\F_q(t)$. We propose an algorithm for constructing a quaternion algebra over $\F_q(t)$ which is split at a place $w$ if and only if $w\notin S$. Such an algebra does exist if and only if the cardinality of $S$ is even (see \cite[Theorem III.3.1]{Vigneras}).

\begin{theorem}\label{quat}
Assume that $S$ is an even finite set of places of $\F_q(t)$. Then there exists a randomized polynomial time algorithm (polynomial in $d$ and $\log~q $) which constructs a quaternion algebra $H$ such that $H\otimes\F_q(t)_v$ is split if and only if $v\notin S$ (here $\F_q(t)_w$ denotes the completion of $\F_q(t)$ at $w$). 
\end{theorem}
\begin{proof}

Let $H(a,b)$ denote the quaternion algebra $(a,b;\F_q(t),-1)$ with parameters $a, b \in \F_q(t)$. We will look for $a$ and $b$ in the form 
$$a=f_1\cdots f_ku, b=\lambda f_1\cdots f_k $$
where the $f_i$ are the finite places in $S$ (which are monic irreducible polynomials), $u$ is an irreducible polynomial in $\F_q[t]$ and $0\neq\lambda\in\F_q$. 
First note that $H(a,b)$ is split at a place $f$ (either a finite place or at infinity) if the equation
$$ax^2+by^2-z^2=0$$
is solvable in the completion $\F_q(t)_{(f)}$ \cite[Project 4, Exercise 3.5]{Tapas}.

If $H(a,b)$ should be split at infinity, then we choose the degree parity $u$ in a way that the degree of $a$ is even and we choose the leading coefficient of $u$ to be 1. We choose $\lambda$ to be 1 (actually we could choose $\lambda$ to be any nonzero element in $\F_q$). By Lemma \ref{Lemma1'}, $H(a,b)$ will be split at infinity. 

If $H(a,b)$ should not be split at infinity, then there are two cases. If the degree of $b$ is odd (we have not chosen a $\lambda$ yet but since $\lambda$ is a nonzero constant it will not influence the degree of $b$), then we choose $u$ in a way that the degree parity of $a$ is even and we choose the leading coefficient of $u$ to be a non-square element in $\F_q$. If the degree of $b$ is even, then we choose $\lambda$ to be a non-square element in $\F_q$ and we choose $u$ in a way that the degree of $a$ is odd (we do not have any constraints on the leading coefficient of $u$, thus we choose it to be 1). By Lemma \ref{Lemma1'}, $H(a,b)$ will be a division algebra at infinity. 

Thus we have imposed degree parity and leading coefficient conditions on $u$ and we have chosen a suitable $\lambda$.

Now we impose conditions on $u$ to ensure that for all $i$ we have that $H(a,b)_{(f_i)}$ is a division algebra. By Lemma \ref{Lemma1} we have that this happens if and only if $-u\lambda$ is not a square mod $f_i$. Note that we have already chosen $\lambda$, thus only choosing $u$ in a suitable way remains. For every $i$ we pick residue classes $\alpha_i$ modulo $f_i$ in the following fashion. If $-\lambda$ is a square modulo $f_i$, then we pick $\alpha_i$ to be a non-square element modulo $f_i$. Thus if $u\equiv \alpha_i~ (mod ~ f_i)$ then $-u\lambda$ is a non-square element mod $f_i$. If $-\lambda$ is not a square modulo $f_i$, then we pick $\alpha_i=1$. By the Chinese remainder theorem there exists a unique residue class $B$ modulo $f_1\cdots f_k$ which satisfies the condition $B\equiv \alpha_i~ (mod ~f_i)$. Thus if 
$$u\equiv B~ (mod ~f_1\cdots f_k),$$
then $H(a,b)_{(f_i)}$ will be a division algebra for all $f_i$.

Now we summarize the steps of the algorithm. Let $F=f_1\cdots f_k$ and let $d=deg~ F$
\begin{enumerate}
\item Choose a $\lambda$, a degree parity $\epsilon$ (which is zero if the degree of $u$ should be even, otherwise it is 1) and a leading coefficient $\mu$ in the way described above.
\item Compute the residue class $B$ modulo $f_1\cdots f_k$ by Chinese remaindering.
\item Pick a random monic polynomial $g$ of degree $3d+\epsilon$. Check if the polynomial $u'=f_1\cdots f_kg+\frac{B}{\mu}$ is irreducible. If $u'$ is irreducible, then let $u=\mu u$. Output $a=f_1\cdots f_ku$ and $b=\lambda f_1\cdots f_k$. If $u'$ is not irreducible, then pick a new $g$. 
\end{enumerate}

The output quaternion algebra $H(a,b)$ doesn't split at any place $w \in S$. Also, it splits at every place except maybe at $u$ by Lemma \ref{Lemma1} and \ref{Lemma1'}. Thus, since the number of places where it does not split is even by \cite[Theorem III.3.1]{Vigneras}, it must split at $u$ as well. 

Finally, we need to show that the algorithm runs in polynomial time. The first three steps are deterministic and run in polynomial time. 

We analyze the last step similarly as in the proof of \cite[Theorem 30]{IKR2}. We have the following inequality due to Proposition \ref{number}:

\begin{equation*}
\left|S_{N}(B,F)-\frac{q^N}{\Phi(F)N}\right|\leq \frac{1}{N}(d+1)q^{\frac{N}{2}}.
\end{equation*} 

Note the we chose $N=4d+\epsilon$, which implies that

$$\frac{S_{N}(B,F)}{q^{N-d}}\geq \frac{q^N}{q^{N-d}\Phi(F)N}-\frac{(d+1)q^{\frac{N}{2}}}{Nq^{N-d}}\geq \frac{1}{N}-\frac{d+1}{Nq^{\frac{N}{2}-d}}\geq \frac{1}{N}-\frac{d+1}{Nq^d}\geq \frac{1}{3N}.$$
This means that the probability that after $3N$ rounds we do not find an irreducible polynomial in the residue class is smaller than $\frac{1}{2}$. Hence this step runs in polynomial time. 
\end{proof}

\subsection{Symbol algebras}

Our next goal is to generalize the algorithm from Theorem \ref{quat} to symbol algebras. 

We first recall some basic facts on symbol algebras which will be useful for the construction of our algorithms. Let $K$ be a field such that $K^*$ contains a cyclic subgroup of order $n$, and take $\epsilon \in K^*$ a primitive $n$--th root of unity. Symbol algebras support the following splitting condition. 

\begin{proposition}\label{splitsymb}\cite[Chapter 11, Corollary 4]{Draxl}.
 The symbol algebra $(a,b;K,\epsilon)$ is split if and only if $b$ is a norm in the extension $K(a^{\frac{1}{n}})|K$. 
\end{proposition}

Proposition \ref{splitsymb} implies that if $a$ is an $n$--th power in $K$, then $(a,b;K,\epsilon)$ splits. We also have the following formula. 
\begin{proposition}\label{mult}\cite[Chapter 11, Lemma 3]{Draxl}
$$(a,b;K,\epsilon)\otimes (a',b;K,\epsilon) \sim (aa',b; K, \epsilon),$$
where $\sim$ denotes Brauer equivalence.
\end{proposition}

In this section we assume that $\F_q$ contains the $nth$ roots of unity, i.e., $q\equiv 1~ (mod~ n)$. Let $\epsilon \in \F_q$ be a primitive $n$--th root of unity.   

\begin{proposition}\label{mod}
Let $f$ be a monic irreducible polynomial in $\F_q[t]$, where $q\equiv 1~(mod~n)$. Let $\epsilon$ be a primitive $n$th root of unity in $\F_q$. Denote by $\F_q(t)_{(f)}$ the completion of $\F_q(t)$ at the place corresponding to $f$. Let $a,a'$ be units in the local ring of $\F_q(t)_{(f)}$, and $b \in \F_q(t)_{(f)}$. Suppose that $a\equiv a'~ (mod~ f)$. Then the symbol $ \F_q(t)_{(f)}$--algebras $(a,b;  \F_q(t)_{(f)},\epsilon)$ and $(a',b; \F_q(t)_{(f)},\epsilon)$ are Brauer equivalent. 
\end{proposition}
\begin{proof}
If $c$ is a unit and $c\equiv 1~ (mod ~f)$, then $c$ is an $n$th power by Hensel's lemma thus the algebra $(c,b; \F_q(t)_{(f)},\epsilon)$ splits. Now one has to observe that $a'a^{-1}\equiv 1~ (mod~ f)$ and that the opposite algebra of $(a,b;  \F_q(t)_{(f)},\epsilon)$ is Brauer equivalent to  $(a^{-1},b; \F_q(t)_{(f)},\epsilon)$.
\end{proof}

We would like to cite a lemma from \cite[Theorem 5]{BG} which provides a formula for calculating Hasse-invariants of cyclic algebras over local fields. The notation $(F,\sigma,a)$ stands for the cyclic $F^{\sigma}$--algebra built from an automorphism $\sigma$ of $F$ of finite order and $a \in F^{\sigma}$. 

\begin{proposition}\label{Hasse}
Let $K$ be a local field (with valuation $v_K$) and let $W$ be an unramified cyclic extension of $K$ of degree $n$. Let $\sigma$ be the unique automorphism of $W$ that reduces to the Frobenius automorphism on residue fields. Then the Hasse invariant of the cyclic algebra $(W, \sigma, b)$ is $\frac{v_K(b)}{n}$. 
\end{proposition}

\begin{remark}
If $\sigma$ reduces to the $k$th power of the Frobenius automorphism, where $k$ is coprime to $n$, then the Hasse invariant of $(W, \sigma, b)$ is $\frac{k'v_K(b)}{n}$ where $kk'\equiv 1~(mod ~n)$ \cite[Chapter 32]{Reiner}.  
\end{remark}

We are now ready to describe the procedure for the construction of a symbol algebra with prescribed Hasse invariants.

\begin{theorem}\label{symbol}
Assume that we are given a set of monic irreducible polynomials $f_1,\dots, f_k$ (in $\F_q[t]$) and a sequence of rational numbers (in reduced form) $\frac{r_1}{s_1},\dots, \frac{r_k}{s_k}, \frac{r_0}{s_0}$. Suppose that the sum of these rational numbers is an integer. Assume that the least common multiple of the $s_i$ is $n$. Then there exists a randomized polynomial time algorithm which constructs a division $\F_q(t)$-algebra $D$, whose local Hasse invariant at $f_i$ is equal $\frac{r_i}{s_i}$, for $i=1, \dots, k$, its local Hasse-invariant at infinity is equal to $\frac{r_0}{s_0}$, and the local Hasse-invariant at every other place is 0. 
\end{theorem}
\begin{proof}

First assume that the degree of $f_1\cdots f_k$ is coprime to $n$.
Let $\epsilon$ be a primitive $n$th root of unity in $\F_q$. Denote the symbol basis of the symbol algebra by $u,v$, i.e., 
$$u^n=a,~ v^n=b, ~ uv=\epsilon vu.$$
We look for $a$ and $b$ in the  form
$$a=s,~ b=f_1\cdots f_k\lambda, $$
where $s$ is a monic irreducible polynomial in $\F_q[t]$ and $\lambda\in\F_q^*$. The algorithm as in Theorem \ref{quat} boils down to choosing $s$ and $\lambda$ in an appropriate way.
First we impose congruence conditions on $s$ in a way that the resulting algebra has Hasse-invariants $\frac{r_i}{s_i}$ at the places $f_i$ for $i =1, \dots, k$. Define the residue class $r_i'$ modulo $n$ such that $\frac{r_i}{s_i}=\frac{r_i'}{n}$. We look at the algebra $D\otimes\F_q(t)_{(f_i)}$. Let $K_i=\F_q[t]/(f_i)$ which is a finite field with $q^{deg~ f_i}$ elements. Note that $C_i=K_i^*/K_i^{*^n}$ is a cyclic group of order $n$. By Proposition \ref{mod} it is enough to find a $\omega_i\in K_i$ such that the symbol algebra $(\omega_i,b; \F_q(t)_{(f_i)}, \epsilon)$ has Hasse invariant $\frac{r_i}{s_i}$ as a central simple $\F_q(t)_{(f_i)}$-algebra. Choose $\delta_i\in K_i$ to be a generator of $C_i$. Then, by Proposition \ref{Hasse} (and the remark after it), $(\delta_i,b, \F_q(t)_{(f_i)}, \epsilon)$ has Hasse invariant $\frac{o_i}{n}$ where $(o_i,n)=1$. Since $(o_i,n)=1$, there exists a residue class $o_i'$ modulo $n$ such that $o_io_i'\equiv r_i'~ (mod ~ n)$. Choose $\omega_i=\delta_i^{o_i'}$. Proposition \ref{mult} implies that $(\omega_i,b; \F_q(t)_{(f_i) \epsilon})$ has Hasse invariant $\frac{r_i}{s_i}$. Thus choose $s$ to be congruent to $\omega_i$ modulo $f_i$ (this imposes $k$ congruence conditions on $s$ which can be made into one using Chinese remaindering as in Theorem \ref{quat}).

Now we impose degree conditions on $s$ and choose $\lambda$ in a way that the resulting symbol algebra has Hasse-invariants $\frac{r_0}{s_0}$ at infinity. Again let $\frac{r_0}{s_0}=\frac{r'}{n}$

If we want the algebra to split at infinity (i.e.,$r=0$), then choose the degree of $s$ to be divisible by $n$. Indeed, then $s$ is an $n$th power by Hensel's lemma in $\F_q((\frac{1}{t}))$, hence by Proposition \ref{splitsymb} the algebra splits. Assume that $r\neq 0$. Let $F=f_1\cdots f_k$ and let $deg~ F\equiv l~ (mod~ n)$. Then choose the degree of $s$ to be congruent to $n-l$ modulo $n$. Then $(uv)^n=\epsilon^{\frac{n(n-1)}{2}}sF\lambda=c$. Note that $c$ is a polynomial whose degree is divisible by $n$ and its leading coefficient is $\lambda$. Let $\mu=\epsilon^{\frac{n(n-1)}{2}}\lambda$. We will now choose an appropriate $\mu\in\F_q$. Then we put $\lambda=(\epsilon^{\frac{n(n-1)}{2}})^{-1}\mu$. Let $w=uv$. Observe that $uw=\epsilon wu$. Now we have the desired unramified extension that splits $D\otimes\F_q((\frac{1}{t}))$, namely the $n$th root of $\mu Fs$. Now we proceed in the same manner as at the finite primes. First choose $\mu_0$ to be a generator of $\F_q^*/\F_q^{*^n}$. Then Proposition \ref{mod} shows that by choosing $\mu=\mu_0$ we get a Hasse invariant $\frac{o}{n}$ at infinity where $(o,n)=1$. Let $o'$ be such that $oo'\equiv r' ~(mod ~n)$. By choosing $\mu=\mu_0^{o'}$ we get the desired Hasse-invariant. 

Now we consider the case where $deg~ F$ is not coprime to $n$. Suppose $deg ~F\equiv l~ (mod~ n)$. Choose an irreducible polynomial $g$ (different from the $f_i$) such that $deg~g\equiv n+1-l~(mod ~n)$. Such a polynomial can be found just by picking a large enough degree and choosing a a polynomial at random. Then we look for $a$ and $b$ in the following form:
$$a=s,~ b=f_1\cdots f_kg\lambda. $$
By implying the same conditions modulo $f_i$ apply on $s$ as in the first part of the proof we guarantee that the local Hasse-invariants at the $f_i$ are $\frac{r_i}{s_i}$. We add the extra condition that $s\equiv 1~ (mod ~g)$. Proposition \ref{splitsymb} implies that $D$ splits at $g$. Finally by choosing $\lambda$ and the degree of $s$ in a suitable way we can achieve that the Hasse-invariant at infinity is $\frac{r}{s}$ as in the first part of the proof (as now the degree of $f_1\cdots f_kg$ is congruent to 1 modulo $n$). 

Note that, by Proposition \ref{splitsymb}, $D$ splits at every finite place different from the $f_i$, as a polynomial over a finite field always has a zero if the number of its variables is greater than its degree (by Chevalley's theorem), and the existence of roots over a local field is reduced to finite fields by Hensel's Lemma. 

We must consider the Hasse invariant at $s$. The Hasse-invariant at $s$ must be zero as the sum of all Hasse-invariants adds up to an integer.

Finally, $D$ is indeed a division algebra as it has index $n$ (it has period $n$ and in the case of global fields, the period equals the index) and is of dimension $n^2$ over $\F_q(t)$.

\end{proof}

\section{Construction of an explicit isomorphism from a simple algebra to its matrix form.}\label{sec:isomorphism}

Let $\A$ be a central simple algebra over $\F_q(t)$ of finite dimension $n^2$.  Let $b_1,\dots,b_{n^2}$ be an $\F_q(t)$-basis of $\A$. Then, for $i, j = 1, \dots n^2$,
$$b_ib_j=\sum_{k=1}^{n^2}\gamma_{ijk}b_k, $$
for $\gamma_{ijk} \in \F_q(t)$. We consider $\A$ to be given as a collection of  \emph{structure constants} $$\{ \gamma_{ijk} : 1 \leq i,j,k \leq n^2 \}.$$ Consider the following problem:
\begin{problem}\label{EI}

Compute an explicit isomorphism of $\F_q(t)$--algebras $\A \cong M_{k}({D})$, for a suitable division $\F_q(t)$--algebra $D$. 
\end{problem}

If the algebra $\A$ is known to be split, then a randomized polynomial time algorithm is proposed in \cite{IKR} which finds an explicit isomorphism $\A\cong M_n(\F_q(t))$. We will first use such a solution to Problem \ref{EI} when $D=\F_q(t)$, in conjunction with \cite{IRS}, to get a randomized polynomial time algorithm which solves the general case, whenever $D$ is known. 

\begin{proposition}\label{AisoM}
Assume that a division $\F_q(t)$--algebra $D$ is given by structure constants and it is known that $\A \cong M_k(D)$. There exists a randomized polynomial time algorithm which computes an explicit isomorphism $\A \cong M_k(D)$.
\end{proposition}
\begin{proof}
First, observe that, from the $\F_q(t)$--basis of $D$ and the structure constants of $D$, one easily gets $m$ and a basis of  $M_k(D)$ with the corresponding structure constants.  Now, we know that  $\A\otimes M_k(D)^{op}\cong M_{n^2}(\F_q(t))$. Using the randomized polynomial time algorithm from \cite{IKR} one can compute an explicit isomorphism $\theta$ between $\A\otimes M_k(D)^{op}$ and $M_{n^2}(\F_q(t))$. Finally,  \cite[Section 4]{IRS} describes a randomized polynomial time method for computing an explicit isomorphism $\phi$ between $\A$ and $M_k(D)$ using $\theta$. 
\end{proof}

We have reduced Problem \ref{EI} to the following one.

\begin{problem}\label{E1}
Let $\A$ be a central simple $\F_q(t)$-algebra of dimension $n^2$ over $\F_q(t)$  given by structure constants. Compute the structure constants of a division $F_q(t)$--algebra $D$ such that $\A \cong M_k(D)$. 
\end{problem}

The algorithm we propose to deal with Problem \ref{E1} rests upon the idea of computing first the local Hasse invariants of $\A$ and, with them at hand, construct a division algebra with the same local Hasse invariants. To this end, we need an algorithm for computing local indices of a central simple algebra over $\F_q(t)$, which is already provided by \cite{Ithesis}.

\begin{lemma}\label{index}\cite[Proposition 6.5.3]{Ithesis}.
There exists a randomized polynomial time algorithm for computing the local index at a given irreducible $f \in \F_q(t)$ of a central simple $\F_q(t)$-algebra $\A$ defined by structure constants.
\end{lemma}
\begin{proof}
The proof of \cite[Proposition 6.3.5]{Ithesis} boils down, in this case, to the following procedure. Compute a maximal $\F_q[t]$-order $\Gamma$ in $\A$ using the algorithm from \cite[Theorem 6.4.2]{Ithesis}. Let $f$ be a monic irreducible polynomial. Then $\Gamma/f\Gamma$ is a finite algebra $C$ over the field $\F_q[t]/(f)$. Then one computes the radical of $C$ using the algorithm from \cite{CIW} and then one computes the factor $C/Rad(C)$. Then the dimension of this radical-free part over its center is the local index at $f$. 
\end{proof}

\begin{proposition}\label{Hasse2}
There exists a randomized polynomial time algorithm which computes the Hasse-invariants of a central simple $\F_q(t)$-algebra $\A$ given by structure constants,  assuming that $gcd(q,n)=1$ where $n$ is the degree of $\A$ over $\F_q(t)$.
\end{proposition}
\begin{proof}
Compute a maximal $\F_q[t]$-order $\Gamma$ in $\A$ using the algorithm from \cite{IKR}. The Hasse-invariant is zero for every monic irreducible polynomial which does not divide the discriminant of $\Gamma$. Thus by factoring the discriminant we have a list of monic irreducible polynomials for which the Hasse-invariant needs to be computed. 

First we propose an algorithm that decides whether the Hasse-invariant of $\A$ at the place $f$ equals $k/n$ or not, for each $k = 0, \dots, n-1$. We choose a finite place $g$ (i.e., a monic irreducible polynomial) which is different from $f$. Using the algorithm from \cite{BG} we construct a division algebra $D$ with Hasse invariants $\frac{n-k}{n}$ at $f$ and $\frac{k}{n}$ at $g$ (this splits at infinity since the sum of the Hasse-invariants is an integer). Using Lemma \ref{index}, we compute the local index of the central simple algebra $\A\otimes D$ at $f$. Since the local Hasse invariants of the tensor product of two central simple algebras add up, the Hasse invariant of $\A$ at $f$ is $\frac{k}{n}$ if and only if the local index at $f$ of $\A\otimes D$ is equal to 1. 

Finally we do this computation for every $k$ and every monic irreducible $f$ dividing the discriminant of $\Gamma$ and we are done. 
\end{proof}

\begin{theorem}\label{isomorphism}
Let $\A$ be a central simple $\F_q(t)$-algebra of dimension $n^2$  given by structure constants.  Assume that $\A$ is split at infinity and that $(n,q)=1$. There exists a randomized polynomial time algorithm for computing a central division $\F_q(t)$--algebra $D$ and an explicit isomorphim $\A \cong M_k(D)$. 
\end{theorem}
\begin{proof}
By Proposition \ref{Hasse}, we can compute the set $S$ of Hasse invariants of $\A$.  For the second step, construct a division algebra $D$ ($D$ should be given by an $\F_q(t)$-basis and structure constants), whose non-zero Hasse-invariants are exactly the elements of the set $S$.  This can be done by the algorithm from \cite{BG}. 
We need to show that the denominator of each nonzero Hasse-invariant is relative prime to $q$. The least common multiple of the $s_i$ is equal to the index of $\A$. Since the index of $\A$ is a divisor of $n$ and $(q,n)=1$, each of the $s_i$ is coprime to $q$. This implies that the algorithm from \cite{BG}
can be applied. Note that this algorithm returns $D$ in a cyclic algebra form, not in a structure constant from. However, from a cyclic algebra representation there exists a polynomial time algorithm which computes structure constants.  Finally, apply the algorithm from Proposition \ref{AisoM}.
\end{proof}

The following consequence of Theorem \ref{isomorphism}  will be used later. 

\begin{corollary}\label{idempotent1}
Let $\A$ be a central simple $\F_q(t)$-algebra of dimension $n^2$ over $\F_q(t)$ given by structure constants.  Assume that $\A$ is split at infinity and that $(n,q)=1$. There exists a randomized polynomial time algorithm for computing a primitive idempotent of $\A$.
\end{corollary}

Actually the conditions for Corollary \ref{idempotent1} can be relaxed. Assume that $\A$ is not split at infinity but it is split at a place corresponding to the monic irreducible polynomial $f(t)=t+c$ where $c\in\F_q$. Let $s=\frac{1}{f}$. Then one has that $\F_q(t)=\F_q(s)$ only now the infinite place of $\F_q(s)$ corresponds to the finite place $f$ of $\F_q(t)$. This shows the following:

\begin{theorem}\label{inf}
Let $\A$ be a central simple $\F_q(t)$-algebra of dimension $n^2$ given by structure constants. Assume that $(n,q)=1$ and that $\A$ is either split at infinity or at a finite place $f$ where $f$ corresponds to a linear polynomial. Then there exists a randomized polynomial time algorithm which finds a primitive idempotent in $\A$,  henceforth, an explicit isomorphism $\A \cong M_k(D)$ for a division $\F_q(t)$--algebra $D$ Brauer equivalent to $\A$. 
\end{theorem}
\begin{proof}
By computing the nonzero Hasse-invariants of $\A$ we obtain a linear polynomial $f(t)=t+c$ at which $\A$ splits. Then let $s=\frac{1}{t+c}$ and rewrite the structure constants of $\A$ in terms of $s$ (every structure constant is a rational function in $s$). Now this new algebra is split at infinity, thus we can find a primitive idempotent $\A$ using Corollary \ref{idempotent1}. Finally substitute $s=\frac{1}{t+c}$ and obtain the primitive idempotent as an $\F_q(t)$-linear combination of the basis elements. Finally, a straightforward argument shows how to get an explicit isomorphism $\A \cong M_k(D)$ from a primitive idempotent of $\A$. 
\end{proof}

Theorem \ref{inf} implies that assuming that the degree of the algebra and $q$ are relatively prime we only encounter a problem if $\A$ is split at every linear place. This is much less restrictive, then the original conditions of Corollary \ref{idempotent1} or Theorem \ref{isomorphism}. In conclusion, Theorem  \ref{inf} solves Problem \ref{EI} completely if $\F_q$ contains the $n$th roots of unity and for ``almost all" central simple $\F_q(t)$-algebras when $\F_q$ does not contain a primitive $n$th root of unity.

\section{Constructing Constacyclic convolutional codes}

In this section we consider skew-constacyclic convolutional codes which are related to skew-cyclic convolutional codes in a similar fashion as linear constacyclic block codes are related to cyclic codes. Our main goal is to construct constacyclic codes of designed Hamming distance and propose a decoding algorithm.

We present cyclic algebras as factor rings of skew polynomial rings, with the aim of making use of the computational tools (e.g. extended Euclidean Algorithm) available for these non-commutative polynomials.  We will need also to consider the more general situation of $K$--linear codes, where  $K$ be a finite extension of $\F_q(t)$, even thought our primary interest is the case $K = \F_q(t)$.  We start by recalling the definition of skew polynomial rings over $K$. 

\begin{definition}
Let $\sigma$ be an automorphism of $K$ of order $n$. Then $R=K[x;\sigma]$ consists of the usual polynomials over $K$ with the standard addition and multiplication induced by the relation 
$xa=\sigma(a)x$, where $a\in K$. 
\end{definition}

Let us denote the fixed field of $\sigma$ by $K^{\sigma}$. Suppose $\lambda\in K^{\sigma}$. Then it is easy to see that the Ore polynomial $x^n-\lambda$ is in the center of $R$, and $\A=K[x,\sigma]/(x^n-\lambda)$ is a cyclic algebra over $K^{\sigma}$ which is  isomorphic, as a $K$--vector space, to $K^n$ by the following map:
$$\mathfrak{v}: \sum_{i=0}^{n-1} a_ix^i \mapsto (a_0,\dots,a_{n-1})\in K^n.$$ 

Thus we can define the Hamming weight of an element in $\A$.
\begin{definition}
The Hamming weight $w(f)$ of an element $f = \sum_{i=0}^{n-1} a_ix^i \in \A$ is the number of nonzero $a_i$. The Hamming distance between $f, g \in \A$ is defined by $d(f,g) = w(f-g)$. 
\end{definition}

Next we define skew-constacyclic codes. 
\begin{definition}
Let $\A=K[x,\sigma]/(x^n-\lambda)$. A skew-constacyclic $K$--linear convolutional code is a left ideal of $\A$ endowed with the Hamming distance. 
\end{definition}

Skew-cyclic convolutional codes from \cite{Gomez/alt:2016} are obtained by setting  $K=\F_q(t), \lambda=1$. The rest of the section will be divided into two subsections. In the first subsection we consider the case where $\lambda$ is a norm in the extension $K|K^{\sigma}$. We are mainly interested in the case when $K=\F_q(t)$. However, we will prove results for general $K$ as well, as we need them in the other subsection.

The second subsection is devoted to the case where $K=\F_q(t)$ and $\lambda$ is not a norm in the extension $\F_q(t)|\F_q(t)^{\sigma}$. Here we start from a primitive idempotent of $\A$. Next we construct a set of maximal idempotents which are permuted by $\sigma^m$ (the starting idempotent is denoted by $e$), where $m$ is the index of $\A$. We consider the left ideal generated $e,\sigma^m(e),\dots,\sigma^{m(k-2)}(e)$ and show that this code has Hamming minimum distance at least $k$ and propose a decoding algorithm. 

\subsection{The norm case}

In this subsection we consider the case where $\lambda$ is a norm in the extension $K|K^{\sigma}$.

\begin{definition}
Let $K$ be a finite extension of $\F_q(t)$ and let $\sigma$ be an automorphism of finite order $n$ of $K$. Then the $j$th norm map $N_j$ is defined in the following way:
$$N_0(x)=1, ~ N_j(x)=x\sigma(x)\cdots \sigma^{j-1}(x)$$
\end{definition}

It is well known that the cyclic algebras $\A=K[x,\sigma]/(x^n-\lambda)$ and $\B=K[y,\sigma]/(y^n-1)$ are isomorphic. The key observation of this subsection is that there is a map from $\A=K[x,\sigma]/(x^n-\lambda)$ to $\B=K[y,\sigma]/(y^n-1)$ which is not only an isomorphism of rings, but is also an isometry with respect to the Hamming distance.
\begin{proposition}\label{map}
Let $\theta$ be the map $\A\rightarrow\B$ defined by
$$\theta: \sum_{i=0}^{n-1}a_ix^i\mapsto \sum_{i=0}^{n-1}a_iN_i(a)y^i,$$
where $N_{K|K^{\sigma}}(a)=\lambda$. Then $\theta$ is an algebra isomorphism which is an isometry with respect to the Hamming distance.
\end{proposition}
\begin{proof}
It is easy to check that $\theta$ is a homomorphism of $K^\sigma$--algebras. It is also an isometry since $N_i(a)\neq 0$, because the norm of $a$ is $\lambda$ (which is nonzero). 
The inverse of $\theta$ is the  map
$$\theta^{-1}: \sum_{i=0}^{n-1}a_iy^i\mapsto \sum_{i=0}^{n-1}a_iN_i(a^{-1})x^i.$$
\end{proof}

First we consider the case when $K=\F_q(t)$. The map $\theta^{-1}$ provides an easy way to construct codes of designed distance $\delta$ from skew codes. First we construct a skew Reed-Solomon code of designed distance $\delta$ in $\B=\F_q(t)[y,\sigma]/(y^n-1)$ (using the method from \cite{GLN2}). Let this code be $C$. By Proposition \ref{map}, $\theta^{-1}(C)$ is a skew-constacyclic code in $\A=\F_q(t)[x,\sigma]/(x^n-\lambda)$. The only thing we need is to be able to solve the norm equation $N_{\F_q(t)|\F_q(t)^{\sigma}}(a)=\lambda$. This can be done using the algorithm from \cite{IKR} since $\F_q(t)^{\sigma}$ is isomorphic to $\F_q(t)$ (by Lüroth's theorem) and such an isomorphism can be computed by the method of \cite{GRS}. The decoding procedure from \cite{GLN2} can also be adjusted. You receive an element $m$ in $\A$. Then apply $\theta$ to $m$ and decode it in $\B$ as $c\in\B$. Finally $\theta^{-1}(c)$ is the decoding of $m$. Naturally, these codes will also be MDS. We summarize these observations in a theorem:
\begin{theorem}
Let $\A=\F_q(t)[x,\sigma]/(x^n-\lambda)$, where $\lambda$ is a norm in the extension $\F_q(t)|\F_q(t)^{\sigma}$. Then there exists a randomized polynomial time algorithm which computes skew-constacyclic MDS codes of designed distance $\delta$ and there also exists a polynomial time decoding algorithm for these codes.
\end{theorem}

These results imply that the norm case is closely related to the skew-cyclic case. The following proposition is a slight generalization of \cite[Theorem 4]{GLN2} which will be needed when dealing with a $\lambda$ which is not a norm. For $f_1, \dots, f_t \in K[x;\sigma]$, the notation $[f_1,\dots,f_t]_\ell$ stands for the least common left multiple of $f_1, \dots, f_t$. 

\begin{proposition}\label{estimate}
Let $K$ be a finite extension of $\F_q(t)$ and let $\A=K[x;\sigma]/(x^n-1)$. Let $\alpha$ generate a normal basis of the extension $K|K^{\sigma}$ and let $\beta=\alpha^{-1}\sigma(\alpha)$. Let $m$ be a divisor of $n$. Then the code generated by 
$$[x-\beta,x-\sigma^m(\beta),\dots,x-\sigma^{m(k-2)}(\beta)]_l$$
has Hamming distance at least $k$.
\end{proposition}
\begin{proof}
The same proof as the proof of Theorem 4 in \cite{GLN2} applies. 
\end{proof}

Moreover, such a code can also be decoded by the same algorithm as described in \cite{GLN2}.

\subsection{The case where $\lambda$ is not a norm}

In this section, we deal with the case where $\lambda$ is not a norm. We assume we know an explicit algebra isomorphism between $\A=\F_q(t)[x,\sigma]/(x^n-\lambda)$ and $M_{n/m}(D)$ where $D$ is a division algebra of index $m$ over $\F_q(t)^{\sigma}$.  Such an isomorphism can be computed by means of the algorithms from Section \ref{sec:isomorphism}. 

The following theorem provides an orthogonal system of primitive idempotents adapted to our purposes. First note that $\sigma$, when applied coefficientwise to an Ore polynomial, is an automorphism of $\A$. By an abuse of notation we will denote this automorphism also by $\sigma$.
\begin{theorem}
Let $\A=\F_q(t)[x,\sigma]/(x^n-\lambda)$ and assume that $\lambda$ is not an $r$th power for every $r$ dividing $n$, and that $(q,n)=1$. Let $m$ be the index of $\A$. Suppose we have an isomorphism between $\A$ and $M_{n/m}(D)$ where $D$ is the division algebra Brauer equivalent to $\A$. Then there exists a randomized polynomial time algorithm which finds a primitive idempotent $e_0$ such that $e_0,\sigma^m(e_0),\dots,\sigma^{n-m}(e_0)$ is an orthogonal system of primitive idempotents in $\A$.  
\end{theorem}
\begin{proof}
We already have an isomorphism between $\A$ and $M_{n/m}(D)$ so, by an abuse of notation, we refer to $x$ as a matrix from $M_{n/m}(D)$. Let $s\in M_{n/m}(D)$ be the matrix with $\lambda$ in the bottom left corner, 1s over the diagonal and zero everywhere else (this is the usual companion matrix of the polynomial $y^{n/m}-\lambda$). Let $K=\F_q(t)^{\sigma}$. The minimal polynomial of $s$ and $x^m$ over $K$ is $y^{n/m}-\lambda\in K[y]$. The polynomial $y^{n/m}-\lambda$ is irreducible over $K$ because $\lambda$ is not an $r$th power by assumption for every $r$ dividing $n$. This implies that $K(s)$ and $K(x^m)$ are subfields of $\A$ which are isomorphic, thus, by the Noether-Skolem theorem, they are conjugate. This means that there exists an element $z\in K(x)$ which is a conjugate of $s$ and $z^{n/m}=\lambda$. Since $(n,q)=1$ there exists a field automorphism of $K(x)$ which maps $z$ to $x^m$. By the Noether-Skolem theorem this field automorphism is also realized by a conjugation. Finally we get that $s$ and $x^m$ are conjugates. An element $h$ can be computed by solving a system of linear equations for which $h^{-1}x^mh=s$. 

Let $f$ be the primitive idempotent in $M_{n/m}(D)$ having 1 in the top left corner and zero everywhere else. Then $f,s^{-1}fs,\dots, s^{1-n/m}fs^{n/m-1}$ is a complete orthogonal system of primitive idempotents. Since $h^{-1}x^mh=s$ we have that 
$$f,h^{-1}x^{-m}hfh^{-1}x^mh,\dots, (h^{-1}x^{-m}h)^{1-n/m}f(h^{-1}x^mh)^{n/m-1}$$
is a complete system of primitive orthogonal idempotents. It is now easy to see that choosing $e_0=hfh^{-1}$ suffices. 
\end{proof}

So, we will assume we are given a primitive  idempotent $e \in \A$ such that 
\[
e_0,\sigma^m(e_0),\dots,\sigma^{n-m}(e_0)
\]
 is an orthogonal system of primitive idempotents in $\A$. Let $e=1-e_0$. Now we are ready to define our code.

\begin{definition}
A skew Reed-Solomon constacyclic convolutional code of designed distance $k\leq \frac{n}{m}$ is defined as the code generated, as a left ideal, by
$$[e,\sigma^m(e),\dots,\sigma^{m(k-2)}(e)]_l. $$
\end{definition}

Now are goal is to justify the previous definition and show that the code has indeed Hamming distance at least $k$.
\begin{theorem}\label{main}
The code $C$ generated by $[e,\sigma^m(e),\dots,\sigma^{m(k-2)}(e)]_l$ has Hamming distance at least $k$ and it also admits a decoding algorithm which runs in polynomial time. 
\end{theorem}

The first key idea of the proof is the construction of an isometric embedding of $\A=\F_q(t)[x,\sigma]/(x^n-\lambda)$ into $\A'=M[x,\phi]/(x^n-\lambda)$ where $M$ is the splitting field of the polynomial $s^n-\lambda\in\F_q(t)[s]$ and $\phi$ is an automorphism of $M$ which restricted to $L$ is $\sigma$. 

\begin{proposition}\label{embed}
Let $\sigma$ be an automorphism of $\F_q(t)$ of order $n$. Let $\lambda\in\F_q(t)^{\sigma}$ and let $M$ be the splitting field of the polynomial $s^n-\lambda\in\F_q(t)[s]$. Then there exists an automorphism $\phi$ of $M$ with the following properties:
\begin{enumerate}
    \item $\phi$ restricted to $\F_q(t)$ is $\sigma$ and $\phi$ has order $n$,
    \item $\lambda$ is a norm in the extension $M|M^{\phi}$.
\end{enumerate}
\end{proposition}
\begin{proof}
We distinguish two cases. First assume that $\F_q$ contains the $n$th roots of unity. Then $M=\F_q(t)(\lambda^{\frac{1}{n}})$. The field $M$ admits an $\F_q(t)$-basis $1,\lambda^{\frac{1}{n}},\dots,\lambda^{\frac{l}{n}}$ where $l=\frac{n}{d}$ (if $d$ is the largest positive integer for which $\lambda$ is a $d$th power where $d$ divides $n$). Then consider the following map:
$$\phi: \mu_0+\mu_1\lambda^{\frac{1}{n}}+\dots+\mu_k\lambda^{\frac{l}{n}}\mapsto \sigma(\mu_0)+\sigma(\mu_1)\lambda^{\frac{1}{n}}+\dots+\sigma(\mu_k)\lambda^{\frac{l}{n}}.$$
The map $\phi$ is an automorphism of $M$ since $\lambda$ is fixed by $\sigma$. Also $\phi$ has order $n$ since its $n$th power is the identity and restricted $\F_q(t)$ it is $\sigma$ which has order $n$ (as an automorphism of $\F_q(t)$). Finally, since $\lambda^{\frac{1}{n}}$ is fixed by $\phi$, we have that $\lambda$ is the norm of $\lambda^{\frac{1}{n}}$ in the extension $M|M^{\phi}$.

Now assume that $\F_q$ does not contain the $n$th roots of unity. Then $M=\F_r(t)(\lambda^{\frac{1}{n}})$ where $\F_r$ is an extension of $\F_q$ by a primitive $n$th root of unity. In this case we first extend $\sigma$ to $\F_r(t)$ in a natural way (the image of $t$ is exactly the same as in $\F_q(t)$). This fixes $\F_r$. Then we extend in the exact same fashion as in the previous case. 
\end{proof}

Proposition \ref{embed} gives us an isometric embedding of $\A=\F_q(t)[x,\sigma]/(x^n-\lambda)$ into $\A'=M[x,\phi]/(x^n-\lambda)$. Actually $\A'$ naturally contains $\A$. The important observation is that $\A'$ is now a full matrix algebra over the field $M^{\phi}$. 

Let $C$ be the code generated by $[e,\sigma^m(e),\dots,\sigma^{m(k-2)}(e)]_l$. Now the element $e$ is contained in $\A'$ as well. Consider the left ideal $L$ of $\A'$ generated by $e$. 
\begin{lemma}\label{maxideal}
The left ideal $L$ is contained in a maximal left ideal generated by $x-\beta$ and such a $\beta$ can be computed in polynomial time. 
\end{lemma}
\begin{proof}
First we show that if we already have a maximal left ideal $I$ containing $L$, then we can compute $\beta$. A maximal left ideal has dimension $n(n-1)$ over $M^{\phi}$. The $M$-subspace generated by $1$ and $x$ has dimension $2n$ over $M^{\phi}$. Thus these two subspaces have a nontrivial intersection (a nonzero  intersecting element is of the form $a_1x+a_2$, where $a_1,a_2\in M$ and $a_1\neq 0$ since otherwise it would be invertible). Now we proceed by proposing an algorithm for finding a maximal left ideal containing $e$. Since we have an element (the element $\lambda^{\frac{1}{n}}$) in $M$ whose norm is $\lambda$ in the extension $M|M^{\phi}$ we can compute an explicit isomorphism between $\A'$ and $M_n(M^{\phi})$ (if one has an element $\mu\in M$ whose norm is $\lambda$ , then $y-\mu$ is a rank 1 element in $\A'$). The element $e$ is diagonalizable with eigenvalues 0 and 1. We compute an eigenbasis and thus a diagonalisation. Let $geg^{-1}$ be the diagonal matrix with 0s and 1s in the diagonal. Let $w$ be a matrix where all the zeros in the diagonal of $geg^{-1}$ are switched to 1s except at one place. Then $w$ generates a maximal left ideal which contains $geg^{-1}$. This implies that the maximal left ideal $g^{-1}wg$ contains $e$. 
\end{proof}

Now we are ready to prove Theorem \ref{main}.
\begin{proof}[Proof of Theorem \ref{main}]
Let us consider the code $C$ generated by $[e,\sigma^m(e),\dots,\sigma^{m(k-2)}(e)]_l$. Let $\A'=M[x,\phi]/(x^n-\lambda)$ as defined in Proposition \ref{embed} and compute $\beta$ as described in Lemma \ref{maxideal}. Let $\alpha$ be an element in $M$ which generates a normal basis of the extension $M|M^{\phi}$. Let $a=\frac{\beta\alpha}{\phi(\alpha)}$ and let $\gamma=\phi(\alpha)\alpha^{-1}$. Now consider the embedding $\theta$ of $\A'$ into $B=M[y,\phi]/(y^n-1)$ defined by:
$$\theta: \sum_{i=0}^{n-1}a_ix^i\mapsto \sum_{i=0}^{n-1}a_iN_i(a)y^i.$$
The maximal left ideal of $\A'$ generated $x-\beta$ maps to the maximal left ideal $y-\frac{\beta}{a}=y-\gamma$. Thus the left ideal $C$ embeds isometrically into the left ideal of $\B$ generated by $$[y-\gamma,y-\phi^m(\gamma),\dots,\phi^{m(k-2)}]_\ell.$$ Proposition \ref{estimate} shows that the Hamming distance of $C$ is at least $k$ (as it is contained in a code which has Hamming distance at least $k$). Decoding also works now in a natural way. We decode the code in $\B$ (this is now a skew-cyclic RS-code). Then we compute its preimage via the map $\theta$ (the method for computing the inverse of $\theta$ is described in the previous subsection). 
\end{proof}

Theorem \ref{main} shows that these constacyclic codes are subcodes of skew-cyclic RS codes over extensions of $\F_q(t)$. The bound we prove on their Hamming distance is tight in the sense that if $\A$ is a division algebra then the Hamming distance of any constacyclic code is 1.

\end{document}